\documentclass[12pt]{article}
\usepackage{amsmath,amssymb,amsthm}
\usepackage{url}
\usepackage{braket}
\begin{document}

\def\RR{\boldsymbol{R}}
\def\s{\sigma}
\def\Epsilon{{\cal E}}
\def\BB {\boldsymbol{B}}
\def\CC {\boldsymbol{C}}
\def\XX{\mathbb{X}}
\def\FF{\boldsymbol{F}}
\def\SS{\boldsymbol{S}}
\def\qq{\boldsymbol{q}}
\def\alp{\alpha}
\def\bet{\beta}
\def\tet{\theta}

\def\Tet{\Theta}
\def\gam{\gamma}
\def\d{\Delta}
\def\eps{\epsilon}
\def\ome{\omega}
\def\nek{,\ldots,}
\def\ox{\overline x}
\def\bfX{{\bf X}}
\def\ve{\vec}
\def\rr{\boldsymbol{r}}
\def\sss{\boldsymbol{s}}
\def\vsig{{\vec\sigma}}
\def\na{\nabla}
\def\sq{\sqrt}
\def\BB {\boldsymbol{B}}
\def\KK {\boldsymbol{K}}
\def\LL {\boldsymbol{L}}
\def\AAA {\boldsymbol{A}}
\def\XX{\mathbb{X}}
\def\alp{\alpha}
\def\bet{\beta}
\def\tet{\theta}
\def\Tet{\Theta}
\def\d{\Delta}
\def\eps{\epsilon}
\def\ome{\omega}
\def\nek{,\ldots,}
\def\ox{\overline x}
\def\bfX{{\bf X}}
\def\ve{\vec}
\def\vsig{{\vec\sigma}}
\def\na{\nabla}
\def\sq{\sqrt}
\def\xx{\boldsymbol{x}}
\def\zz{\boldsymbol{z}}
\def\dd{\boldsymbol{d}}
\def\ee{\boldsymbol{e}}
\def\ff{\boldsymbol{f}}
\def\yy{\boldsymbol{y}}
\def\bb{\boldsymbol{b}}
\def\ee{\boldsymbol{e}}
\def\uu{\boldsymbol{u}}
\def\vv{\boldsymbol{v}}
\def\ww{\boldsymbol{w}}
\def\cc{\boldsymbol{c}}
\def\sss{\boldsymbol{s}}
\def\tt{\boldsymbol{t}}
\def\UU{\boldsymbol{U}}
\def\DD{\boldsymbol{D}}
\def\EE{\boldsymbol{E}}
\def\VV{\boldsymbol{V}}
\def\GG{\boldsymbol{G}}
\def\HH{\boldsymbol{H}}
\def\WW{\boldsymbol{W}}
\def\PP{\boldsymbol{P}}
\def\pp{\boldsymbol{p}}
\def\QQ{\boldsymbol{Q}}
\def\RR{\boldsymbol{R}}
\def\TT{\boldsymbol{T}}
\def\XX{\boldsymbol{X}}
\def\YY{\boldsymbol{Y}}
\def\ZZ{\boldsymbol{Z}}
\def\MM{\boldsymbol{M}}
\def\NN{\boldsymbol{N}}
\def\aaa{\boldsymbol{a}}
\def\gg{\boldsymbol{g}}
\def\hh{\boldsymbol{h}}
\def\II{\boldsymbol{I}}
\def\JJ{\boldsymbol{J}}
\def\OO{\boldsymbol{O}}
\def\00{\boldsymbol{0}}
\def\11{\boldsymbol{1}}
\def\ggamma{\mbox{\boldmath{$\gamma$}}}
\def\ddelta{\mbox{\boldmath{$\delta$}}}
\def\eepsilon{\mbox{\boldmath{$\epsilon$}}}
\def\rrho{\mbox{\boldmath{$\rho$}}}
\def\xxi{\mbox{\boldmath{$\xi$}}}
\def\eeta{\mbox{\boldmath{$\eta$}}}
\def\LLambda{\mbox{\boldmath{$\Lambda$}}}
\def\SSigma{\mbox{\boldmath{$\Sigma$}}}
\def\pphi{\mbox{\boldmath{$\phi$}}}
\def\ppsi{\mbox{\boldmath{$\psi$}}}
\def\PPsi{\mbox{\boldmath{$\Psi$}}}
\newcommand{\off}[1]{}
\def\M{\scriptscriptstyle{\boldsymbol{M}}}

\renewcommand{\theequation}{\thechapter.\arabic{equation}}
\newcommand{\mrm}{\mathrm}
\newcommand{\beq}{\begin{equation}}
\newcommand{\eeq}{\end{equation}}
\newcommand{\bthm}{\begin{theorem}}
\newcommand{\ethm}{\end{theorem}}
\newcommand{\ben}{\begin{enumerate}}
\newcommand{\een}{\end{enumerate}}
\newcommand {\non}{\nonumber}
\newcommand{\C}{\mbox{$\mathbb{C}$}}
\newcommand{\rank}{\mbox{rank}}

\newtheorem{theorem}{Theorem}
\newtheorem{lemma}[theorem]{Lemma}
\newtheorem{corollary}[theorem]{Corollary}
\newtheorem{definition}[theorem]{Definition}
\newtheorem{example}[theorem]{Example}
\newtheorem*{definition*}{Definition}
\pagenumbering{roman}
\renewcommand{\thelemma}{\thesection.\arabic{lemma}}
\renewcommand{\thetable}{\thesection.\arabic{table}}
\renewcommand{\thetable}{\arabic{table}}
\renewcommand{\thedefinition}{\thesection.\arabic{definition}}
\renewcommand{\theexample}{\thesection.\arabic{example}}
\renewcommand{\theequation}{\thesection.\arabic{equation}}
\newcommand{\mysection}[1]{\section{#1}\setcounter{equation}{0}
\setcounter{theorem}{0} \setcounter{lemma}{0}
\setcounter{definition}{0}}

\title
{\bf On a Vectorized Version of a  Generalized Richardson  Extrapolation Process}

\author
{Avram Sidi\\
Computer Science Department\\
Technion - Israel Institute of Technology\\ Haifa 32000, Israel\\~ \\
E-mail:\ \ \url{asidi@cs.technion.ac.il}\\
URL:\ \ \url{http://www.cs.technion.ac.il/~asidi}}
\date{Appeared in: \ {\em Numerical Algorithms,} 74:937--949, 2017}
\bigskip\bigskip
\maketitle
\thispagestyle{empty}
\newpage
\begin{abstract}\noindent
Let   $\{\xx_m\}$ be a vector sequence that satisfies
$$ \xx_m\sim \sss+\sum^\infty_{i=1}\alpha_i \gg_i(m)\quad\text{as $m\to\infty$},$$
$\sss$ being the limit or antilimit of $\{\xx_m\}$ and
$\{\gg_i(m)\}^\infty_{i=1}$ being an asymptotic scale as $m\to\infty$, in the sense that $$\lim_{m\to\infty}\frac{\|\gg_{i+1}(m)\|}{\|\gg_{i}(m)\|}=0,\quad i=1,2,\ldots.$$ The vector sequences $\{\gg_i(m)\}^\infty_{m=0}$, $i=1,2,\ldots,$ are known, as well as $\{\xx_m\}$.
In this work, we analyze the convergence and convergence acceleration properties of a vectorized version of the generalized Richardson extrapolation process that is defined via the equations
$$ \sum^k_{i=1}\braket{\yy,\Delta\gg_{i}(m)}\widetilde{\alpha}_i=\braket{\yy,\Delta\xx_m},\quad n\leq m\leq n+k-1;\quad \sss_{n,k}=\xx_n+\sum^k_{i=1}\widetilde{\alpha}_i\gg_{i}(n),$$ $\sss_{n,k}$ being the approximation to $\sss$. Here $\yy$ is some  nonzero vector, $\braket{\cdot\,,\cdot}$ is an inner product, such that $\braket{\alpha\aaa,\beta\bb}=\overline{\alpha}\beta\braket{\aaa,\bb}$, and $\Delta\xx_m=\xx_{m+1}-~\xx_m$ and $\Delta\gg_i(m)=\gg_i(m+1)-\gg_i(m)$. By imposing a minimal number of reasonable additional conditions on the $\gg_i(m)$, we show that the error $\sss_{n,k}-\sss$ has a full asymptotic expansion as $n\to\infty$. We also  show that actual convergence acceleration takes place and we provide a complete classification of it.
\end{abstract}

\vspace{1cm} \noindent {\bf Mathematics Subject Classification
2010:} 65B05; 65B10; 40A05; 40A25.

\vspace{1cm} \noindent {\bf Keywords and expressions:}
acceleration of convergence; vector extrapolation methods; vectorized generalized Richardson extrapolation process.

\thispagestyle{empty}
\newpage
\pagenumbering{arabic}

\section{Introduction}
Let $\mathbb{X}$ be a  finite or infinite dimensional linear inner product space with the inner product   $\braket{\cdot\,,\cdot}$  defined  such that $\braket{\alpha\aaa,\beta\bb}=\overline{\alpha}\beta\braket{\aaa,\bb}$, and let $\|\cdot\|$ be the norm induced by this inner product, namely, $\|\zz\|=\sqrt{\braket{\zz,\zz}}$.

Let   $\{\xx_m\}$ be a vector sequence in  $\mathbb{X}$, and let $\xx_m$ have  an asymptotic expansion of the form

\beq \label{eqmju10a}\xx_m\sim \sss+\sum^\infty_{i=1}\alpha_i\gg_i(m) \quad\text{as $m\to\infty$},\eeq
$\sss$ being the limit or antilimit of $\{\xx_m\}$ and
 $\{\gg_i(m)\}^\infty_{i=1}$ being an asymptotic scale as $m\to\infty$, in the sense  that
\beq \label{eqmju10b}\lim_{m\to\infty}
\frac{\|\gg_{i+1}(m)\|}{\|\gg_i(m)\|}=0,\quad i=1,2,\ldots\ .\eeq
 The vector sequences $\{\gg_i(m)\}^\infty_{m=0}$, $i=1,2,\ldots,$ are known, as well as $\{\xx_m\}$. The scalars $\alpha_i$ do not have to be known.
By \eqref{eqmju10a}, we mean
\beq \label{eqmju10aa}\bigg\|\xx_m- \sss-\sum^r_{i=1}\alpha_i\gg_i(m) \bigg\|=o(\|\gg_r(m)\|)\quad\text{as $m\to\infty$},\quad\forall\ r\geq1.\eeq
Of course, the summation $\sum^\infty_{i=1}\alpha_i\gg_i(m)$ in  the asymptotic expansion of \eqref{eqmju10a} does not need to be convergent; it may diverge in general.
Finally, the $\alpha_i$ are not all nonzero necessarily; some may be zero in general.\footnote{We may think of an Euler--Maclaurin expansion that may not be full, for example.}

Clearly, if $\alpha_1\neq0$ and $\lim_{m\to\infty}\gg_1(m)=\00$, then $\{\xx_m\}$ converges and we have  $\lim_{m\to\infty}\xx_m=\sss$. If $\alpha_1\neq0$ and
$\lim_{m\to\infty}\gg_1(m)$ does not exist, then $\{\xx_m\}$ diverges.
In case it converges, the convergence of the sequence $\{\xx_m\}$ can be accelerated via a suitable extrapolation method, which will produce good approximations to $\sss$.  Extrapolation methods can be very useful for obtaining good approximations to $\sss$ also in case of divergence, at least in some cases.

 In this work, we would like to analyze the convergence and acceleration properties of one such method, namely, the {\em vector E-algorithm} of Brezinski \cite{Brezinski:1980:GEA}. See also Brezinski and Redivo Zaglia \cite[Chapter 4, pp. 228--232]{Brezinski:1991:EMT}.
The vector E-algorithm produces from the sequence $\{\xx_m\}$ and the sequences $\{\gg_i(m)\}^\infty_{m=0}$, $i=1,2,\ldots,$ a two-dimensional array of approximations
$\sss_{n,k}$, which are defined via
\beq \label{eqmju17} \sss_{n,k}=\xx_n-\sum^{k}_{i=1}\widetilde{\alpha}_i \gg_i(n),\eeq
 the $\widetilde{\alpha}_i$  being the solution to the $k\times k$ linear system
    \beq \label{eqmju18} \sum^k_{i=1}\braket{\yy,\Delta\gg_i(m)}\widetilde{\alpha}_i=\braket{\yy,\Delta\xx_m},\quad m=n,n+1,\ldots,n+k-1. \eeq
 Here  $\yy$ is a  nonzero vector in $\mathbb{X}$, $\Delta\xx_m=\xx_{m+1}-\xx_m$ and $\Delta\gg_i(m)=\gg_i(m+1)-\gg_i(m)$.
Taken together, and by Cramer's rule, \eqref{eqmju17} and \eqref{eqmju18} give rise to
the following determinant representation for $\sss_{n,k}$:
\beq \label{eqmju19} \sss_{n,k}=
f_{n,k}(\xx); \quad \xx\equiv \{\xx_m\},\eeq where, for an arbitrary vector sequence $\vv\equiv\{\vv_m\}$ in $\mathbb{X}$,
\beq \label{eqmju20}
f_{n,k}(\vv)=\frac{N_{n,k}(\vv)}{D_{n,k}}=\frac{\begin{vmatrix}\vv_n & \braket{\yy,\Delta \vv_n}& \cdots &\braket{\yy,\Delta \vv_{n+k-1}}\\
\gg_1(n) & \braket{\yy,\Delta \gg_1(n)}& \cdots &\braket{\yy,\Delta \gg_1(n+k-1)}\\ \vdots &\vdots && \vdots \\
\gg_k(n) & \braket{\yy,\Delta \gg_k(n)}& \cdots &\braket{\yy,\Delta \gg_k(n+k-1)}\end{vmatrix}}
{\begin{vmatrix}  \braket{\yy,\Delta \gg_1(n)}& \cdots &\braket{\yy,\Delta \gg_1(n+k-1)}\\ \vdots && \vdots \\
\braket{\yy,\Delta \gg_k(n)}& \cdots &\braket{\yy,\Delta \gg_k(n+k-1)}\end{vmatrix}}.
\eeq
Of course, we are assuming that $D_{n,k}$, the denominator determinant of $\sss_{n,k}$, is nonzero.
Note also that $N_{n,k}(\xx)$, the numerator determinant of $\sss_{n,k}$, which is a vector, is to be interpreted as its expansion with respect to its first column.

A recursion relation for the $\sss_{n,k}$ is given in Brezinski \cite{Brezinski:1980:GEA}. Different recursion relations for this method are also given in  Ford and Sidi \cite{Ford:1988:RAV}.

For convenience, let us arrange the $\sss_{n,k}$  in a two-dimensional array as in Table~\ref{table:snk}, where $\sss_{n,0}=\xx_n$, $n=0,1,\ldots\ .$

\begin{table}[h]
\[ \begin{array}{cccc}
 \sss_{0,0}&\sss_{1,0} & \sss_{2,0}& \cdots\\
\sss_{0,1}&\sss_{1,1} & \sss_{2,1}& \cdots\\
\sss_{0,2}&\sss_{1,2} & \sss_{2,2}& \cdots\\
\vdots&\vdots&\vdots&\ddots
  \end{array} \]

\caption {\label{table:snk} The extrapolation table.}
\end{table}

\section{A convergence theory}
\setcounter{equation}{0}
Convergence acceleration properties of the rows $\{\sss_{n,k}\}^\infty_{n=0}$, $k=1,2,\ldots,$  of the extrapolation table, that is,
convergence acceleration properties of $\sss_{n,k}$ as $n\to\infty$ with $k$ fixed,  have been considered  under different  conditions in the works of  Wimp \cite[Chapter 10, p. 180, Theorem 1]{Wimp:1981:STA} and Matos \cite{Matos:1991:ARV}. Here we provide a new study, whose results are summarized in  Theorem \ref{th:FGREP} that is stated and proved below. This theorem   provides optimal results in the form of
\begin{enumerate}
\item  a genuine asymptotic expansion for $\sss_{n,k}$ as $n\to\infty$, and
\item   a definitive and quantitative convergence acceleration result.
\end{enumerate}
The  technique we use to prove Theorem \ref{th:FGREP} is derived in part from Wimp \cite{Wimp:1981:STA} and mostly from Sidi \cite{Sidi:1990:GRE}, with necessary modifications to accommodate vector sequences.
It also involves the notion of  {\em generalized asymptotic expansion}; see   Temme \cite[Chapter 1]{Temme:2015:AMI}, for example. For convenience, we give the precise definition of this notion here.

\begin{definition*} Let $\{\phi_i(m)\}^\infty_{i=1}$ and $\{\psi_i(m)\}^\infty_{i=1}$ be two asymptotic scales as $m\to\infty$. Let also $\{W_m\}^\infty_{m=0}$ be a given sequence. We say that the formal series $\sum^\infty_{i=1}a_i\phi_i(m)$ is the {\em generalized asymptotic expansion of $W_m$ with respect to $\{\psi_i(m)\}^\infty_{i=1}$ as $m\to\infty$,} written in the form
$$  W_m\sim \sum^\infty_{i=1}a_i\phi_i(m)\quad \text{as $m\to\infty$};\quad \{\psi_i\},$$ provided
$$ W_m-\sum^r_{i=1}a_i\phi_i(m)=o(\psi_r(m))\quad \text{as $m\to\infty$,}\quad \forall\ r\geq1.$$
\end{definition*}

The notation we use in the sequel  is precisely that introduced in the previous section.

\begin{theorem}\label{th:FGREP} Let the sequence $\{\xx_m\}$ be as in
\eqref{eqmju10a}, with the $\gg_i(m)$ satisfying \eqref{eqmju10b},  and
\beq \label{eqmju21}\begin{split}  \lim_{m\to\infty}\frac{\braket{\yy,\gg_i(m+1)}}{\braket{\yy,\gg_i(m)}}=b_i\neq 1,\quad i=1,2,\ldots,\\
b_i\  \text{distinct};\quad |b_1|>|b_2|>\cdots;\quad \lim_{i\to\infty}b_i=0,\end{split}
\eeq
in addition. Assume also that

\beq \label{eqmju22a}\lim_{m\to\infty}\frac{\gg_i(m)}{\braket{\yy,\Delta\gg_i(m)}}=\widehat{\gg}_i\neq\00,\quad i=1,2,\ldots,\eeq
and define
\beq \label{eqmju22b}\widehat{\hh}_{k,i}=
\begin{vmatrix}\widehat{\gg}_i& 1& b_i&\cdots&b_i^{k-1}\\
\widehat{\gg}_1& 1& b_1&\cdots&b_1^{k-1}\\ \vdots&\vdots& \vdots &&\vdots\\
\widehat{\gg}_k& 1& b_k&\cdots&b_k^{k-1}\end{vmatrix},\quad i\geq k+1.\eeq
Then the following are true:
\begin{enumerate}
\item There holds
\beq \label{eqmju23}  \lim_{m\to\infty}\frac{\braket{\yy,\Delta\gg_i(m+1)}}{\braket{\yy,\Delta\gg_i(m)}}=b_i,
\quad i=1,2,\ldots,\eeq
in addition to \eqref{eqmju21}. Furthermore,
the sequence $\{\braket{\yy,\Delta\gg_i(m)}\}^\infty_{i=1}$ is an asymptotic scale as $m\to\infty$, that is,
\beq \label{eqmju22c} \lim_{m\to\infty}\frac{\braket{\yy,\Delta\gg_{i+1}(m)}}{\braket{\yy,\Delta\gg_i(m)}}=0,\quad i=1,2,\ldots\ .\eeq

\item
With arbitrary $\vv\equiv \{\vv_m\}^\infty_{m=0}$,  $f_{n,k}(\vv)$ defined in \eqref{eqmju20} exist for all $n\geq n_0$, $n_0$ being some positive integer independent of $\vv$.
\item
 \begin{enumerate}
 \item
 With $\gg_i\equiv \{\gg_i(m)\}^\infty_{m=0}$, we have
   $f_{n,k}(\gg_i)=~\00$ for $i=1,\ldots,k,$ while for $i\geq k+1$,
\beq \label{eqmju55}\begin{split}
\frac{ f_{n,k}(\gg_i)}{\braket{\yy,\Delta\gg_i(n)}}&\sim \frac{\widehat{\hh}_{k,i}}{V(b_1,\ldots,b_k)}\quad\text{as $n\to\infty$},\quad \text{if $\widehat{\hh}_{k,i}\neq \00$},
\\
\frac{ f_{n,k}(\gg_i)}{\braket{\yy,\Delta\gg_i(n)}}&=o(1)\quad\text{as $n\to\infty$},\quad \text{if $\widehat{\hh}_{k,i}=\00$}, \end{split}
\eeq and also
\beq  \label{eqmju55a}\begin{split}
  \|f_{n,k}(\gg_i)\|&\sim C_{k,i}\|\gg_i(n)\|\quad\text{as $n\to\infty$},
  \quad\text{if $\widehat{\hh}_{k,i}\neq \00$}, \\
  \|f_{n,k}(\gg_i)\|&=o(\|\gg_i(n)\|)\quad\text{as $n\to\infty$},
  \quad\text{if $\widehat{\hh}_{k,i}=\00$},\end{split}
  \eeq
  where
\beq  \label{eqmju55ak} \quad C_{k,i}=\frac{1}{|V(b_1,\ldots,b_k)|} \frac{\|\widehat{\hh}_{k,i}\|}{\|\widehat{\gg}_i\|}\eeq
and $V(c_1,\ldots, c_k)$ is the Vandermonde determinant of $c_1,\ldots,c_k$, given as in
\beq \label{Vander}V(c_1,\ldots,c_k)=\begin{vmatrix} 1& c_1&\cdots&c_1^{k-1}\\ 1& c_2 &\cdots&c_2^{k-1}\\ \vdots& \vdots&& \vdots\\ 1& c_k&\cdots&c_k^{k-1}\end{vmatrix}=\prod_{1\leq i<j\leq k}(c_j-c_i).\eeq
\item
In addition, for $i\geq k+1,$
$\{f_{n,k}(\gg_i)\}^\infty_{i=k+1}$ is an asymptotic scale as $n\to\infty$, in the following generalized sense:

\beq  \label{eqmju10k}\lim_{n\to\infty}\frac{\|f_{n,k}(\gg_{i+1})\|}{\|\gg_{i}(n)\|}=0, \quad i\geq k+1.\eeq
\end{enumerate}
\item $\sss_{n,k}$ has a  genuine generalized asymptotic expansion with respect to the asymptotic scale $\{\gg_i(n)\}^\infty_{i=1}$ as $n\to\infty$; namely,
 \beq \label{eqmju31}\sss_{n,k}\sim\sss+\sum^\infty_{i=k+1}\alpha_i f_{n,k}(\gg_i)\quad\text{as $n\to\infty$};\quad \{\gg_i\}, \eeq
 in the sense that
 \beq \label{eqmju31a}\sss_{n,k}-\sss-\sum^{r}_{i=k+1}\alpha_i f_{n,k}(\gg_i)=o(\gg_r(n))\quad\text{as $n\to\infty$},\quad \forall \ r\geq k+1.\eeq
 We also have
  \beq \label{eqmju31e}\sss_{n,k}-\sss-\sum^{r}_{i=k+1}\alpha_i f_{n,k}(\gg_i)=o(f_{n,k}(\gg_r))\quad\text{as $n\to\infty$},\quad \text{if $\widehat{\hh}_{k,r}\neq\00$}. \eeq
 \item
 Let  $\alpha_{k+\mu}$ be  the first nonzero $\alpha_{k+i}$ with $i\geq k+1$. Then the following are true:
 \begin{enumerate}
 \item $\sss_{n,k}$  satisfies
 \beq \label{eqmju31t} \sss_{n,k}-\sss=O(\gg_{k+\mu}(n))\quad\text{as $n\to\infty$}, \eeq
 and, therefore,  also
 \beq \label{eqmju31p} \sss_{n,k+j}-\sss=O(\gg_{k+\mu}(n))\quad\text{as $n\to\infty$}, \quad
 j=0,1,\ldots,k+\mu-1.\eeq
 \item
 We also have
 \beq \label{eqmju31c}\sss_{n,k}-\sss\sim\alpha_{k+\mu}f_{n,k}(\gg_{k+\mu})\quad\text{as $n\to\infty$},\quad\text{if $\widehat{\hh}_{k,k+\mu}\neq\00$}.
 \eeq
   As a result, provided $\widehat{\hh}_{k+j,k+\mu}\neq\00$, $j=0,1,\ldots,\mu-1$, we also have
 \beq \label{eqmju31f}\sss_{n,k+j}-\sss\sim\alpha_{k+\mu}f_{n,k+j}(\gg_{k+\mu})\quad\text{as $n\to\infty$},\quad j=0,1,\ldots,\mu-1, \eeq which also implies
 \beq \label{eqmju31d} \|\sss_{n,k+j}-\sss\| \sim |\alpha_{k+\mu}|\,C_{k+j,k+\mu}\,\|\gg_{k+\mu}(n)\| \quad\text{as $n\to\infty$}, \quad j=0,1,\ldots,\mu-1. \eeq
 \item
 If  $\alpha_k\neq0$ and $\widehat{\hh}_{k-1,k}\neq\00$, then
 \beq \label{eqmju31b}\frac{\|\sss_{n,k+j}-\sss\|}{\|\sss_{n,k-1}-\sss\|}
 =O\bigg(\frac{\|\gg_{k+\mu}(n)\|}{\|\gg_{k}(n)\|}\bigg)=o(1)\quad\text{as $n\to\infty$},\quad
 j=0,1,\ldots,k+\mu-1.\eeq
\end{enumerate}

  \end{enumerate}
\end{theorem}
\begin{proof}
\noindent{\em Proof of part 1:}
 We first note that
 $$ \frac{\braket{\yy,\Delta\gg_i(m+1)}}{\braket{\yy,\Delta\gg_i(m)}}=
 \frac{\braket{\yy,\gg_i(m+1)}}{\braket{\yy,\gg_i(m)}}
 \frac
 {\displaystyle\frac{\braket{\yy,\gg_i(m+2)}}{\braket{\yy,\gg_i(m+1)}}-1}
 {\displaystyle\frac{\braket{\yy,\gg_i(m+1)}}{\braket{\yy,\gg_i(m)}}-1}.
 $$
 Taking now limits as $m\to\infty$, and invoking \eqref{eqmju21}, we obtain
\eqref{eqmju23}.

Next, by  \eqref{eqmju22a}, we have the asymptotic equality
\beq \label{eqmju58a}\gg_i(m)\sim \braket{\yy,\Delta \gg_i(m)}\,\widehat{\gg}_i \quad\text{as $m\to\infty$},\eeq which, upon taking norms, gives the asymptotic equality
\beq \label{eqmju58b}|\braket{\yy,\Delta \gg_i(m)}|\sim\frac{\| \gg_i(m)\|}{\|\widehat{\gg}_i\|}\quad \text{as $m\to\infty$}. \eeq
 Therefore,
$$ \frac{|\braket{\yy,\Delta \gg_{i+1}(m)}|}{|\braket{\yy,\Delta \gg_i(m)}|}\sim
\frac{\|\widehat{\gg}_{i}\|}{\|\widehat{\gg}_{i+1}\|}
\frac{\|\gg_{i+1}(m)\|}{\|\gg_{i}(m)\|}\quad\text{as $m\to\infty$}.$$
Invoking now the fact that $\{\gg_i(m)\}^\infty_{i=1}$  itself is an asymptotic scale as $m\to\infty$, as in \eqref{eqmju10b}, the result in
\eqref{eqmju22c} follows.
\medskip

\noindent{\em Proof of part 2:}
By \eqref{eqmju20}, $f_{n,k}(\vv)$ for arbitrary $\vv\equiv\{\vv_m\}$ exists provided $D_{n,k}$, the denominator determinant, is nonzero. Therefore, we need to  analyze only  the  determinant $D_{n,k}$ in \eqref{eqmju20}.
Let us set
\beq \label{eqmju24}\eta_{i,j}(m)=\frac{\braket{\yy,\Delta\gg_i(m+j)}}{\braket{\yy,\Delta\gg_i(m)}},\quad i,j=1,2,\ldots, \eeq and observe  that
\beq \label{eqmju25e}\eta_{i,j}(m)=\prod^j_{r=1}\frac{\braket{\yy,\Delta\gg_i(m+r)}}
{\braket{\yy,\Delta\gg_i(m+r-1)}}. \eeq
Letting $m\to\infty$ and invoking \eqref{eqmju23}, we obtain
\beq  \label{eqmju25} \lim_{m\to\infty}\eta_{i,j}(m)=b_i^j. \eeq
Factoring out $\braket{\yy,\Delta\gg_j(n)}$ from the $j$th row of $D_{n,k}$, $j=1,\ldots,k,$ we have
\beq \label{eqmju25f}\frac{D_{n,k}}{\prod^k_{j=1} \braket{\yy,\Delta\gg_j(n)}}=
\begin{vmatrix}1&\eta_{1,1}(n)&\cdots&\eta_{1,k-1}(n)\\
\vdots&\vdots&&\vdots\\ 1&\eta_{k,1}(n)&\cdots&\eta_{k,k-1}(n)\end{vmatrix}\equiv\psi_{n,k},\eeq
which, upon letting $n\to\infty$, gives
\beq \label{eqmju26}\lim_{n\to\infty}\psi_{n,k}=V(b_1,b_2,\ldots, b_k)=
\prod_{1\leq i<j\leq k}(b_j-b_i),\eeq
 this limit being nonzero since the $b_i$ are distinct.
Therefore,
$$D_{n,k}\sim V(b_1,b_2,\ldots, b_k)\prod^k_{j=1} \braket{\yy,\Delta\gg_j(n)}\quad\text{as $n\to\infty$}.$$
From this and from \eqref{eqmju58b}, we conclude that $D_{n,k}\neq0$ for all large $n$.
Since $D_{n,k}$ is also independent of $\vv$, we have that $D_{n,k}\neq0$ for all $n\geq n_0$, $n_0$ being independent of $\vv$ trivially.
\medskip

\noindent{\em Proof of part 3:}
We now turn to $f_{n,k}(\gg_i)=N_{n,k}(\gg_i)/D_{n,k}$, where
\beq \label{eqmju30}N_{n,k}(\gg_i)=
\begin{vmatrix}\gg_i(n)& \braket{\yy,\Delta \gg_i(n)}& \cdots &\braket{\yy,\Delta \gg_i(n+k-1)}\\
\gg_1(n) & \braket{\yy,\Delta \gg_1(n)}& \cdots &\braket{\yy,\Delta \gg_1(n+k-1)}\\ \vdots &\vdots && \vdots \\
\gg_k(n) & \braket{\yy,\Delta \gg_k(n)}& \cdots &\braket{\yy,\Delta \gg_k(n+k-1)}\end{vmatrix}.\eeq
We first observe  that $N_{n,k}(\gg_i)=\00$ for $i=1,\ldots,k,$ since the determinant in \eqref{eqmju30} has two identical rows when $1\leq i\leq k$. This proves that $f_{n,k}(\gg_i)=\00$ for $i=1,\ldots,k.$ Therefore, we consider the case  $i\geq k+1$.
Proceeding as in the analysis of $D_{n,k}$, let us  factor out $\braket{\yy,\Delta \gg_i(n)}$ and $\braket{\yy,\Delta \gg_1(n)}$,\ldots,$\braket{\yy,\Delta \gg_k(n)}$ from the $k+1$ rows of $N_{n,k}(\gg_i)$. We obtain

\beq \label{eqmju33} \frac{N_{n,k}(\gg_i)}{\braket{\yy,\Delta \gg_i(n)}\prod^k_{j=1}\braket{\yy,\Delta \gg_j(n)}}=
\begin{vmatrix}\frac{\gg_i(n)}{\braket{\yy,\Delta \gg_i(n)}}&1&\eta_{i,1}(n)&\cdots&\eta_{i,k-1}(n)\\
\frac{\gg_1(n)}{\braket{\yy,\Delta \gg_1(n)}}&1&\eta_{1,1}(n)&\cdots&\eta_{1,k-1}(n)\\
\vdots&\vdots&\vdots&&\vdots\\
\frac{\gg_k(n)}{\braket{\yy,\Delta \gg_k(n)}}&1&\eta_{k,1}(n)&\cdots&\eta_{k,k-1}(n)\end{vmatrix}\equiv\hh_{k,i}(n), \eeq
which, upon letting $n\to\infty$ and invoking \eqref{eqmju25}, \eqref{eqmju22a}, and \eqref{eqmju22b}, gives
\beq \lim_{n\to\infty}\hh_{k,i}(n)=\widehat{\hh}_{k,i}.\eeq
Combining  now \eqref{eqmju33} with \eqref{eqmju25f}, we obtain
\beq f_{n,k}(\gg_i)=\frac{\braket{\yy,\Delta \gg_i(n)}\hh_{k,i}(n)}{\psi_{n,k}},\eeq  which, upon letting $n\to\infty$, gives
\beq
\lim_{n\to\infty}\frac{f_{n,k}(\gg_i)}{\braket{\yy,\Delta \gg_i(n)}}=\frac{\widehat{\hh}_{k,i}} {V(b_1,\ldots,b_k)},\eeq
  from which, \eqref{eqmju55} follows. \eqref{eqmju55a} is obtained by taking norms in \eqref{eqmju55} and by making use of \eqref{eqmju58b}.

Finally,
$$ \frac{\|f_{n,k}(\gg_{i+1})\|}{\|\gg_{i}(n)\|}=\frac{|\braket{\yy,\Delta \gg_{i+1}(n)}|\, \|\hh_{k,i}(n)\|}
{|\psi_{n,k}|\, \|\gg_{i}(n)\|}, $$ which, upon letting $n\to\infty$ and invoking \eqref{eqmju58b}, gives

$$ \frac{\|f_{n,k}(\gg_{i+1})\|}{\|\gg_{i}(n)\|}\sim\frac{1}{|V(b_1,\ldots,b_k)|}
\frac{\|\hh_{k,i}(n)\|}{\|\widehat{\gg}_{i+1}\|}\frac{\|\gg_{i+1}(n)\|}{\|\gg_{i}(n)\|}\quad \text{as $n\to\infty$}.$$ Invoking here \eqref{eqmju10b}, and noting that $\|\hh_{k,i}(n)\|$ is bounded in $n$, we obtain \eqref{eqmju10k}.
\medskip

\noindent{\em Proof of part 4:}
We now turn to $\sss_{n,k}$. First, we note that
\beq \label{eqmju27}\sss_{n,k}-\sss=f_{n,k}(\xx)-\sss=\frac{N_{n,k}(\xx-\sss)}{D_{n,k}};\quad
\xx-\sss\equiv\{\xx_m-\sss\}, \eeq
with
\beq \label{eqmju28} N_{n,k}(\xx-\sss)=\begin{vmatrix}\xx_n-\sss & \braket{\yy,\Delta \xx_n}& \cdots &\braket{\yy,\Delta \xx_{n+k-1}}\\
\gg_1(n) & \braket{\yy,\Delta \gg_1(n)}& \cdots &\braket{\yy,\Delta \gg_1(n+k-1)}\\ \vdots &\vdots && \vdots \\
\gg_k(n) & \braket{\yy,\Delta \gg_k(n)}& \cdots &\braket{\yy,\Delta \gg_k(n+k-1)}\end{vmatrix},\eeq
 because the coefficient of $\xx_n$ in the expansion of $N_{n,k}(\xx)$ is $D_{n,k}$.
By \eqref{eqmju10a}, the elements in the first row of $N_{n,k}(\xx-\sss)$ have the asymptotic expansions
\begin{gather}\xx_n-\sss\sim  \sum^\infty_{i=1}\alpha_i\gg_i(n)\quad \text{as $n\to\infty$},\notag\\
\braket{\yy,\Delta\xx_{n+j}}\sim\sum^\infty_{i=1}\alpha_i\braket{\yy,\Delta\gg_i(n+j)}\quad\text{as $n\to\infty$},\quad
 j=0,1,\ldots\ .\notag\end{gather}
Multiplying the $(i+1)$st row of $N_{n,k}(\xx-\sss)$ in \eqref{eqmju28} by $\alpha_i$ and subtracting from the first row, $i=1,\ldots,k,$ we obtain

\begin{multline} \label{eqmju29q} N_{n,k}(\xx-\sss)\sim\\
\begin{vmatrix}\sum^\infty_{i=k+1}\alpha_i\gg_i(n) & \sum^\infty_{i=k+1}\alpha_i\braket{\yy,\Delta \gg_i(n)}& \cdots &\sum^\infty_{i=k+1}\alpha_i\braket{\yy,\Delta \gg_i(n+k-1)}\\
\gg_1(n) & \braket{\yy,\Delta \gg_1(n)}& \cdots &\braket{\yy,\Delta \gg_1(n+k-1)}\\ \vdots &\vdots && \vdots \\
\gg_k(n) & \braket{\yy,\Delta \gg_k(n)}& \cdots &\braket{\yy,\Delta \gg_k(n+k-1)}\end{vmatrix}\end{multline}
as  $n\to\infty$.  Taking   the summations $\sum^\infty_{i=k+1}$ and the multiplicative factors $\alpha_i$ from the first row outside the determinant in \eqref{eqmju29q},  we have

\beq \label{eqmju29}N_{n,k}(\xx-\sss)\sim\sum^\infty_{i=k+1}\alpha_iN_{n,k}(\gg_i)\quad\text{as $n\to\infty$};\quad \gg_i\equiv\{\gg_i(m)\}^\infty_{m=0}, \eeq
with $N_{n,k}(\gg_i)$ as in \eqref{eqmju30}.
Substituting \eqref{eqmju29} in \eqref{eqmju27}, we obtain
 the asymptotic expansion of $\sss_{n,k}$ given in \eqref{eqmju31}.
This asymptotic expansion will be a valid  generalized asymptotic expansion with respect to the asymptotic scale $\{\gg_i(n)\}^\infty_{i=1}$ as $n\to\infty$, provided
\beq \label{eqmju32}\sss_{n,k}-\sss-\sum^r_{i=k+1}\alpha_if_{n,k}(\gg_i)=o(\gg_r(n))\quad\text{as $n\to\infty$},\quad \forall \ r\geq k+1.\eeq

 By \eqref{eqmju10aa}, for arbitrary $r$, we have
\beq \label{eqmju79} \xx_m=\sss+\sum^r_{i=1}\alpha_i\gg_i(m)+\eepsilon_{r}(m);\quad \eepsilon_{r}(m)=o(\gg_{r}(m))\quad \text{as $m\to\infty$.}\eeq
Let us substitute this in \eqref{eqmju27} and proceed exactly as above; we obtain
\beq \label{eqmju97} \sss_{n,k}=\sss+\sum^r_{i=k+1}\alpha_if_{n,k}(\gg_i)+ f_{n,k}(\eepsilon_r).\eeq
Comparing \eqref{eqmju97}  with \eqref{eqmju32}, we realize that \eqref{eqmju32} will be satisfied provided
\beq\label{eqmju97c}f_{n,k}(\eepsilon_r)=o(\gg_{r}(n))\quad \text{as $n\to\infty$}.\eeq

Now, $f_{n,k}(\eepsilon_r)=N_{n,k}(\eepsilon_r)/D_{n,k},$ and
$$ N_{n,k}(\eepsilon_r)=\begin{vmatrix}\eepsilon_r(n)& \braket{\yy,\Delta \eepsilon_r(n)}& \cdots &\braket{\yy,\Delta \eepsilon_r(n+k-1)}\\
\gg_1(n) & \braket{\yy,\Delta \gg_1(n)}& \cdots &\braket{\yy,\Delta \gg_1(n+k-1)}\\ \vdots &\vdots && \vdots \\
\gg_k(n) & \braket{\yy,\Delta \gg_k(n)}& \cdots &\braket{\yy,\Delta \gg_k(n+k-1)}\end{vmatrix}.$$
Let us factor out $\braket{\yy,\Delta \gg_{r}(n)}$ and $\braket{\yy,\Delta \gg_1(n)}$,\ldots,$\braket{\yy,\Delta \gg_k(n)}$ from the $k+1$ rows of this determinant. We obtain

\begin{multline}\label{eqmju89}
\frac{N_{n,k}(\eepsilon_r)}{\braket{\yy,\Delta \gg_{r}(n)}\prod^k_{j=1}\braket{\yy,\Delta \gg_j(n)}}=\\
\begin{vmatrix}\frac{\eepsilon_r(n)}{\braket{\yy,\Delta \gg_{r}(n)}}&\frac{\braket{\yy,\Delta \eepsilon_r(n)}}{\braket{\yy,\Delta \gg_{r}(n)}}&\frac{\braket{\yy,\Delta \eepsilon_r(n+1)}}{\braket{\yy,\Delta \gg_{r}(n)}}&\cdots&\frac{\braket{\yy,\Delta \eepsilon_r(n+k-1)}}{\braket{\yy,\Delta \gg_{r}(n)}}\\
\frac{\gg_1(n)}{\braket{\yy,\Delta \gg_1(n)}}&1&\eta_{1,1}(n)&\cdots&\eta_{1,k-1}(n)\\
\vdots&\vdots&\vdots&&\vdots\\
\frac{\gg_k(n)}{\braket{\yy,\Delta \gg_k(n)}}&1&\eta_{k,1}(n)&\cdots&\eta_{k,k-1}(n)\end{vmatrix}\equiv\phi_{n,k}(\eepsilon_r).
\end{multline}
Dividing now \eqref{eqmju89} by \eqref{eqmju25f}, we obtain
$$ f_{n,k}(\eepsilon_r)=\frac{\phi_{n,k}(\eepsilon_r)}{\psi_{n,k}}\braket{\yy,\Delta \gg_r(n)},$$
which, upon taking norms and invoking \eqref{eqmju58b}, gives
$$ \|f_{n,k}(\eepsilon_r)\|\sim\frac{\|\phi_{n,k}(\eepsilon_r)\|}{|V(b_1,\ldots,b_k)|}\frac{\|\gg_r(n)\|}
{\|\widehat{\gg}_r\|}
\quad \text{as $n\to\infty$.}$$
Therefore, \eqref{eqmju32} will hold provided  $\lim_{n\to\infty}\phi_{n,k}(\eepsilon_r)=\00$.
As we already know, with the exception of the elements in the  first row, all the remaining elements
of the determinant $\phi_{n,k}(\eepsilon_r)$  have finite limits as $n\to\infty$,  by \eqref{eqmju22a} and \eqref{eqmju25}. Therefore, $\lim_{n\to\infty}\phi_{n,k}(\eepsilon_r)=\00$ will hold provided all the elements in the  first row of $ \phi_{n,k}(\eepsilon_r)$  tend to zero as $n\to\infty$. That this is the case is what we show next.

First, by   \eqref{eqmju58b}--\eqref{eqmju25}, as $n\to\infty$,
\beq \label{eqmju43} \|\gg_r(n+j)\|\sim\|\widehat{\gg}_r\|\,|\braket{\yy,\Delta \gg_{r}(n+j)}|\sim
|b_r^j|\,\|\widehat{\gg}_r\|\,|\braket{\yy,\Delta \gg_{r}(n)}|\sim |b_r^j|\,\|{\gg}_r(n)\|.\eeq
Next,  by applying the Cauchy--Schwarz inequality to $\braket{\yy,\Delta \eepsilon_{r}(n+j)}$,
and invoking \eqref{eqmju79} and \eqref{eqmju43}, we have
\begin{align*}|\braket{\yy,\Delta \eepsilon_{r}(n+j)}|&\leq \|\yy\|\,(\|\eepsilon_{r}(n+j+1)\|+\|\eepsilon_{r}(n+j)\|) \\
&=o(\|\gg_{r}(n+j+1)\|)+o(\|\gg_{r}(n+j)\|) \quad\text{as $n\to\infty$}\\
&=o(\|\gg_r(n)\|)\quad\text{as $n\to\infty$}.\end{align*}
Invoking also \eqref{eqmju58b},
 for the elements in  the first row of $\phi_{n,k}(\eepsilon_r)$, we finally obtain
$$ \bigg\|\frac{\eepsilon_r(n)}{\braket{\yy,\Delta \gg_{r}(n)}}\bigg\|=o(1)\quad\text{as $n\to\infty$},$$ and
$$\bigg|\frac{\braket{\yy,\Delta \eepsilon_r(n+j)}}{\braket{\yy,\Delta \gg_{r}(n)}}\bigg|=o(1)
\quad\text{as $n\to\infty$},\quad j=0,1,\ldots,k-1.$$
This implies that $\phi_{n,k}(\eepsilon_r)=o(1)$ {as $n\to\infty$}, and the proof is complete.
\medskip

\noindent{\em Proof of Part 5:} By  $\alpha_{k+j}=0$, $j=1,\ldots,\mu-1,$ and
$\alpha_{k+\mu}\neq0$,\footnote{Note that this already takes into account the possibility that $\alpha_{k+1}\neq0$, in which case, $\mu=1$.} the validity of \eqref{eqmju31t} is obvious.
\eqref{eqmju31p} follows from \eqref{eqmju31t}. The validity of \eqref{eqmju31c}--\eqref{eqmju31d}
can be shown in the same way. As for \eqref{eqmju31b}, we start with
$$\frac{\|\sss_{n,k}-\sss\|}{\|\sss_{n,k-1}-\sss\|}\sim \frac{1}{C_{k-1,k}|\alpha_k|}\frac{\|\sss_{n,k}-\sss\|} {\|\gg_{k}(n)\|}, $$
which follows from \eqref{eqmju31d}, and invoke \eqref{eqmju31p}.
We leave the details
 to the reader. \hfill\end{proof}

\section{Remarks on the convergence theory}
\begin{enumerate}
\item
Note that,  Theorem \ref{th:FGREP} is stated under a minimal number of conditions  on the $\gg_i(m)$ and the $\xx_m$. Of these, the condition in \eqref{eqmju21} is already in \cite[p. 180, Eq.(3)]{Wimp:1981:STA}, while that in \eqref{eqmju22a} is a modification of \cite[p. 180, Eq.(5)]{Wimp:1981:STA}.

\item
    The conditions we have imposed on the $\gg_i(m)$ enable us to proceed with the proof rigorously by employing asymptotic equalities $\sim$ everywhere possible. This should be contrasted with  bounds formulated in terms of the big  $O$ notation, which do not allow us to obtain the optimal results we have in our theorem.\footnote{Recall that $u_m\sim v_m$ as $m\to\infty$ if and only if $\lim_{m\to\infty}(u_m/v_m)=1$. One big advantage of asymptotic equalities is that they allow symmetry and division. That is, if $u_m\sim v_m$  then $v_m\sim u_m$  as well.     In addition, $u_m\sim v_m$ and $u_m'\sim v_m'$  also imply $u_m/u_m'\sim v_m/v_m'$.
    On the other hand, if $u_m=O(v_m)$, we do not necessarily have $v_m=O(u_m)$. In addition, $u_m=O(v_m)$ and $u_m'=O(v_m')$  do not necessarily imply $u_m/u_m'=O(v_m/v_m')$.}
\item
Note that we have imposed essentially two different conditions on the $\gg_i(m)$, namely
\eqref{eqmju21} and \eqref{eqmju22a}. One may  naturally think that these  conditions could contradict each other. In addition, one may think that they   could also contradict the very first and fundamental property in  \eqref{eqmju10b}, which must hold  to make \eqref{eqmju10a} a genuine asymptotic expansion.
 Thus, we need to make sure that there are  no contradictions present in our theorem. For this, it is enough to show that  all three conditions can hold simultaneously, which is the case when
 $$\gg_i(m)\sim\ww_ib_i^m\quad \text{as $m\to\infty$},\quad |b_{i}|>|b_{i+1}|\quad \forall\ i\geq1.$$ It is easy to verify that   \eqref{eqmju10b}, \eqref{eqmju21}, and \eqref{eqmju22a} are satisfied simultaneously in this case.
\item
Due to the possibility that  $\widehat{\hh}_{k,i}=\00$ for some  $i\geq k+1$, we cannot claim a priori that
$\{f_{n,k}(\gg_i)\}^\infty_{i=k+1}$ is an asymptotic scale in the regular sense.
Note, however,  that we can  safely replace \eqref{eqmju55} by
$$\|f_{n,k}(\gg_i)\|=O(\|\gg_i(n)\|)\quad \text{as $n\to\infty$},\quad \forall\ i\geq k+1,$$  whether $\widehat{\hh}_{k,i}\neq \00$ or $\widehat{\hh}_{k,i}= \00$.
\item
$\widehat{\hh}_{k,i}\neq\00$ for all  $i\geq k+1$ if, for example,
 the vectors $\widehat{\gg}_i$ are all linearly independent, which is possible if $\mathbb{X}$ is an infinite dimensional space. This can be seen by expanding the determinant defining $\widehat{\hh}_{k,i}$ in \eqref{eqmju22b} with respect to its first column and realizing that $\widehat{\hh}_{k,i}=c_i\widehat{\gg}_i+\sum^k_{j=1}c_j\widehat{\gg}_j$,
where $c_i$ and the $c_j$ are all nonzero Vandermonde determinants.
In such a case,  by \eqref{eqmju55a} and \eqref{eqmju10b},
$$ \frac{\|f_{n,k}(\gg_{i+1})\|}{\|f_{n,k}(\gg_i)\|}\sim \frac{C_{k,i+1}}{C_{k,i}}\frac{\|\gg_{i+1}(n)\|}
{\|\gg_{i}(n)\|}=o(1)\quad \text{as $n\to\infty$},$$ hence
$\{f_{n,k}(\gg_i)\}^\infty_{i=1}$ is an asymptotic scale in the regular sense.  Therefore, the asymptotic expansion of $\sss_{n,k}$ in \eqref{eqmju31} is a {\em regular} asymptotic expansion, which means that
$$\sss_{n,k}-\sss-\sum^{r}_{i=k+1}\alpha_i f_{n,k}(\gg_i)=o(f_{n,k}(\gg_r))\quad\text{as $n\to\infty$}, \quad \forall\ r\geq k+1.$$

\item
When $\alpha_1\neq0$, the sequence $\{\xx_m\}$ is convergent if $|b_1|<1$; it is divergent if $|b_1|\geq1$. The asymptotic result in \eqref{eqmju31a}, which is always true,  shows clearly that $\sss_{n,k}$ converges to $\sss$ faster than $\xx_n$ when $\{\xx_m\}$ is convergent. In case $\{\xx_m\}$ is divergent,
by the assumption  that $\lim_{i\to\infty}b_i=0$, we have that $|b_i|<1$, $i\geq p$, for some integer $p$, and $\sss_{n,k}$ converges when $k\geq p$.

\item Consider the case
$$\alpha_k\neq0,\quad \alpha_{k+1}=\cdots=\alpha_{k+\mu-1}=0,\quad \alpha_{k+\mu}\neq 0.$$
By \eqref{eqmju31a}--\eqref{eqmju31b}, the following transpire:
\begin{itemize}
\item
Whether $ \widehat{\hh}_{k-1,k}=\00$ or $ \widehat{\hh}_{k-1,k}\neq\00$,
$$\sss_{n,k-1}-\sss=O(\gg_k(n))\quad \text{as $n\to\infty$},$$
\item
 Whether $ \widehat{\hh}_{k+j,k+\mu}=\00$ or $ \widehat{\hh}_{k+j,k+\mu}\neq\00$,  $0\leq j \leq\mu-1$, $$\sss_{n,k+j}-\sss=O(\gg_{k+\mu}(n))\quad
\text{as $n\to\infty$}, \quad 0\leq j\leq \mu-1.$$
\item
If $ \widehat{\hh}_{k-1,k}\neq\00$, then
$\sss_{n,k}$ converges faster (or diverges slower) than $\sss_{n,k-1}$,  that is,
$$ \lim_{n\to\infty}\frac{\|\sss_{n,k}-\sss\|}{\|\sss_{n,k-1}-\sss\|}=0.$$
\item
If
$  \widehat{\hh}_{k+j,k+\mu}\neq\00,\quad 0\leq j\leq \mu-1,$
then
$$ \|\sss_{n,k+j}-\sss\|\sim M_{k+j}\|\gg_{k+\mu}(n)\|\quad \text{as $n\to\infty$},\quad j=0,1,\ldots,\mu-1,$$ for some positive constants $M_{k+j}$.  That is,
$\sss_{n,k},\sss_{n,k+1},\ldots,\sss_{n,k+\mu-1}$ converge (or diverge) at precisely the same rate.
\end{itemize}

\item  We have assumed that $\mathbb{X}$ is an inner product space only for the sake of simplicity. We can assume $\mathbb{X}$ to be a normed Banach space in general. In this case, we replace $\braket{\yy, \uu}$ by $Q(\uu)$, where $Q$ is a bounded linear functional on $\mathbb{X}$. With this, the analysis of this section goes through in a straightforward manner.
\end{enumerate}

%\bibliographystyle{plain}
%\bibliography{../../tex/bib/references}

\end{document}